\documentclass[11pt,a4paper]{amsart}

\usepackage[T1]{fontenc}
\usepackage{lmodern}
\usepackage{amsmath,amssymb,amsthm,mathtools}
\usepackage[margin=1in]{geometry}

\newtheorem{theorem}{Theorem}
\newtheorem{lemma}[theorem]{Lemma}
\newtheorem{corollary}[theorem]{Corollary}

\theoremstyle{definition}
\newtheorem{definition}[theorem]{Definition}
\theoremstyle{remark}
\newtheorem{remark}[theorem]{Remark}

\newcommand{\N}{\mathbb N}
\newcommand{\m}{\mathrm{m}}
\newcommand{\fo}{\mathrm{fo}}
\newcommand{\Cyl}{\operatorname{Cyl}}
\newcommand{\dom}{\operatorname{dom}}

\title[One or Infinitely Many Finite-One Degrees]
{Every Nonrecursive Many-One Degree Contains Either One or Infinitely
Many Finite-One Degrees}

\author{Patrizio Cintioli}
\address{Mathematics Division, School of Science and Technology, University of Camerino, Italy}
\email{patrizio.cintioli@unicam.it}

\subjclass[2020]{Primary 03D30; Secondary 03D25.}
\keywords{many-one degrees, finite-one degrees, many-one reducibility, finite-one reducibility,
priority constructions, cylinders, computably enumerable sets}

\begin{document}

\begin{abstract}
This paper proves that every nonrecursive many-one degree contains
either exactly one or infinitely many finite-one degrees. This gives a
negative answer to Open Question 2 of Richter, Stephan, and
Zhang~\cite[p.~17]{RSZ}. Earlier work established that, for
almost every $A\subseteq\N$ with respect to the standard product
measure on $2^\N$, the degree $[A]_{\m}$ contains an infinite antichain
of finite-one degrees~\cite{Cintioli}.
Thus the question had already
received an almost-sure negative answer; the present theorem settles
it for every nonrecursive many-one degree.
\end{abstract}

\maketitle

\section{Introduction}

Finite-one reducibility is the restriction of many-one reducibility in
which every fibre of the reducing function is finite. 
Thus, for all sets $X,Y$,
\[
   X\leq_1 Y\ \Longrightarrow\ X\leq_{\fo}Y\
   \Longrightarrow\ X\leq_{\m}Y.
\]
Given a many-one degree $\mathbf d$, one may therefore ask how many
finite-one degrees occur inside $\mathbf d$, and how they are ordered.

Richter, Stephan, and Zhang asked whether there exists a nonrecursive
many-one degree containing at least two but only finitely many
finite-one degrees~\cite[p.~17]{RSZ}.
Earlier work established that the many-one degree of every
$m$-rigid set contains an infinite antichain of finite-one degrees
\cite{Cintioli}. Since the class of $m$-rigid sets has
Lebesgue measure $1$ and is comeager, this gave an almost-sure and
comeager negative answer to the question.
The present paper removes the $m$-rigidity hypothesis and settles the
question in full: no such degree exists. More precisely, every
nonrecursive many-one degree contains either exactly one or infinitely
many finite-one degrees.

For every set $A$, the cylinder

\[
   \Cyl(A)=\{\langle x,n\rangle:x\in A,\ n\in\N\}
\]

belongs to the same many-one degree as $A$, and
$[\Cyl(A)]_{\fo}$ is the greatest finite-one degree in that many-one
degree. Maslova proved that every nonrecursive computably enumerable (c.e.) many-one degree
contains either exactly one or infinitely many finite-one
degrees~\cite{Maslova1979}. We shall also use the same conclusion when
the complement of a representative is c.e., by passing to the
complementary many-one degree and using complementation.

The main new ingredient is the following local interpolation theorem.
If $A$ is nonrecursive, is not a cylinder, and neither $A$ nor
$\overline A$ is c.e., then there exists $C\equiv_{\m} A$ such that

\[
   [A]_{\fo}<[C]_{\fo}<[\Cyl(A)]_{\fo}.
\]

Thus, under these hypotheses, $[A]_{\fo}$
can be strictly interpolated below the greatest finite-one degree.

The dichotomy follows quickly. Suppose that a nonrecursive many-one
degree contained finitely many but more than one finite-one degrees.
Let $T$ be its greatest finite-one degree, and choose $E<T$ which is
covered by $T$, that is, such that there is no finite-one degree $D$
with

\[
   E<D<T.
\]

Such an $E$ exists because there are only finitely many finite-one
degrees. Choose $A$ representing $E$. Then $A$ is not a cylinder.
If $A$ is c.e., or if $\overline A$ is c.e., Maslova's theorem gives
a contradiction, in the second case after complementation. Otherwise
the local interpolation theorem produces

\[
   E<[C]_{\fo}<T,
\]

again a contradiction.

The proof of the local interpolation theorem is based on a priority
construction relative to $0'$. We construct a set
$N\in\Sigma^0_2$ and let

\[
   P=\N\setminus N.
\]

A profile-realization lemma then yields a total computable surjection
$p$ whose infinite-fibre profile is exactly $P$, and we set

\[
   C=p^{-1}(A).
\]

Two families of requirements control the two strict inequalities.
The upper requirements place suitable finite rows inside $N$; after
profile realization, the corresponding values have finite
$p$-fibres. The lower requirements preserve markers outside $N$ and outside the
domains of partial ascending self-reductions of $A$; these markers
therefore become points of $P$.
The construction uses $0'$ only to decide fixed $\Pi^0_1$ stabilization
and nonconvergence questions; in particular, it never uses $0''$ to
decide whether an arbitrary c.e.\ set is finite.

\section{Preliminaries}

We assume familiarity with the basic notions and notation of
computability theory. Standard references are
Odifreddi~\cite{OdifreddiI,OdifreddiII},
Rogers~\cite{Rogers}, and Soare~\cite{Soare}.

We identify each set with its characteristic function.

Fix a standard effective enumeration

\[
   (W_e)_{e\in\N}
\]

of all c.e.\ subsets of $\N$.

Whenever $(X_n)_{n\in\N}$ is a uniformly c.e.\ family, we use

\[
   X_{n,0}\subseteq X_{n,1}\subseteq\cdots,
   \qquad
   X_n=\bigcup_{s\in\N}X_{n,s},
\]

to denote a standard uniformly computable finite-stage approximation,
with each $X_{n,s}$ finite.

Fix a computable bijective pairing function
\[
   \langle\cdot,\cdot\rangle:\N^2\to\N
\]
with computable projections $\pi_1$ and $\pi_2$.

For sets $X,Y\subseteq\N$, write $X\leq_{\m}Y$ if there is a total
computable function $f$ such that
\[
   x\in X\iff f(x)\in Y
\]
for every $x$. Write $X\leq_1Y$ if such a reduction can be chosen
injective, and write $X\leq_{\fo}Y$ if such a reduction can be chosen
so that every fibre
\[
   f^{-1}(y)
\]
is finite.

For $\rho\in\{\m,1,\fo\}$, write
\[
   X\equiv_\rho Y
   \quad\Longleftrightarrow\quad
   X\leq_\rho Y\text{ and }Y\leq_\rho X.
\]
Write
\[
   X<_\rho Y
   \quad\Longleftrightarrow\quad
   X\leq_\rho Y\text{ and }Y\not\leq_\rho X.
\]
The \emph{$\rho$-degree} of $X$ is the equivalence class
\[
   [X]_\rho=\{Y\subseteq\N:Y\equiv_\rho X\}.
\]
The $\rho$-degrees are ordered by
\[
   [X]_\rho\leq[Y]_\rho
   \quad\Longleftrightarrow\quad
   X\leq_\rho Y.
\]
For the corresponding strict order, we have
\[
   [X]_\rho<[Y]_\rho
   \quad\Longleftrightarrow\quad
   X<_\rho Y.
\]
Here $\leq$ and $<$ denote the induced non-strict and strict orders on $\rho$-degrees, respectively.

In particular, $[X]_{\m}$ and $[X]_{\fo}$ denote the many-one degree
and the finite-one degree of $X$, respectively.

A many-one degree is called \emph{c.e.} if it contains a c.e.\ set,
and \emph{recursive} if it contains a recursive set.

For $A\subseteq\N$, let
\[
   \Cyl(A)=\{\langle x,n\rangle:x\in A,\ n\in\N\}.
\]
We call $A$ a \emph{cylinder} if
\[
   A\equiv_1\Cyl(A).
\]

The maps
\[
   x\longmapsto\langle x,0\rangle
   \qquad\text{and}\qquad
   \langle x,n\rangle\longmapsto x
\]
show that
\[
   A\equiv_{\m}\Cyl(A).
\]

Moreover, if $B\in[A]_{\m}$ and $f$ is a many-one reduction of
$B$ to $A$, then
\[
   x\longmapsto\langle f(x),x\rangle
\]
is a one-one reduction of $B$ to $\Cyl(A)$. Hence
$[\Cyl(A)]_{\fo}$ is the greatest finite-one degree inside
$[A]_{\m}$.

If $p:\N\to\N$ is total computable, define its infinite-fibre
profile by
\[
   I_p=\{y:|p^{-1}(y)|=\infty\}.
\]

A \emph{section} of a map $p:\N\to\N$ is a map $s:\N\to\N$
such that
\[
   p(s(y))=y
   \qquad\text{for every }y\in\N.
\]
Thus a section is a right inverse of $p$.

The following elementary realization lemma allows us to prescribe
exactly which values of a computable surjection have infinite fibres.

\begin{lemma}[Profile realization]\label{lem:profile}
For every $P\in\Pi^0_2$ there is a total computable surjection
$p:\N\to\N$ such that
\[
   I_p=P.
\]
Moreover, $p$ has a computable section.
\end{lemma}

\begin{proof}
Choose a computable predicate $Q$ such that
\[
   y\in P
   \iff
   \forall j\,\exists s\,Q(y,j,s).
\]
For each $k\in\N$, let
\[
   H_k=
   \{y:\forall j<k\,\exists s\,Q(y,j,s)\}.
\]
The sequence $(H_k)_{k\in\N}$ is uniformly c.e., and
\[
   \N=H_0\supseteq H_1\supseteq H_2\supseteq\cdots.
\]
Moreover,
\[
   P=\bigcap_{k\in\N}H_k.
\]

Define
\[
   R=
   \{\langle 0,y\rangle:y\in\N\}
   \cup
   \{\langle k+1,y\rangle:k\in\N,\ y\in H_k\}.
\]
Since $R$ is an infinite c.e.\ set, it admits a total computable
bijection
\[
   r:\N\to R,
\]
obtained by enumerating $R$ without repetitions. Define
\[
   p(n)=\pi_2(r(n)).
\]
Then $p$ is total computable. Since $r$ is onto $R$ and
$\langle 0,y\rangle\in R$ for every $y$, the function $p$ is
surjective.

For every $y$,
\[
   |p^{-1}(y)|
   =
   1+\bigl|\{k:y\in H_k\}\bigr|.
\]
If $y\in P$, then $y\in H_k$ for every $k$, so $p^{-1}(y)$ is
infinite. If $y\notin P$, then $y\notin H_k$ for all sufficiently
large $k$, because the sequence $(H_k)_{k\in\N}$ is decreasing. Hence
$p^{-1}(y)$ is finite. Therefore
\[
   I_p=P.
\]

Finally, define
\[
   s(y)=\mu n\,[p(n)=y].
\]
The function $s$ is total computable because $p$ is computable and
surjective, and
\[
   p(s(y))=y
\]
for every $y$. Thus $s$ is a computable section of $p$.

Moreover, if $s(y)=s(z)$, then
\[
   y=p(s(y))=p(s(z))=z.
\]
Thus $s$ is injective.
\end{proof}

\section{Two combinatorial lemmas}

The next two lemmas will be used in the priority construction. The
first guarantees finite homogeneous rows avoiding
any prescribed finite set of restraints, while the second shows that a
partial ascending self-reduction of a nonrecursive non-cylinder set
must have infinitely many points outside its domain.

A uniformly c.e.\ family $(V_x)_{x\in\N}$ is called
\emph{$A$-homogeneous} if
\[
   z\in V_x \Longrightarrow A(z)=A(x).
\]

\begin{lemma}[Finite-marker avoidance]\label{lem:avoid}
Assume that $A$ is not a cylinder and that neither $A$ nor $\overline A$
is c.e. If $(V_x)_{x\in\N}$ is a uniformly c.e.\ $A$-homogeneous
family, then for every finite $Z\subseteq\N$ there is $x$ such that
\[
   V_x \text{ is finite}
   \qquad\text{and}\qquad
   V_x\cap Z=\varnothing.
\]
\end{lemma}

\begin{proof}

Fix a finite set $Z\subseteq\N$, and suppose, toward a contradiction,
that there is no $x$ such that
\[
   V_x\text{ is finite}
   \qquad\text{and}\qquad
   V_x\cap Z=\varnothing.
\]
Equivalently,
\[
   V_x\cap Z=\varnothing
   \quad\Longrightarrow\quad
   V_x\text{ is infinite}
\]
for every $x$.

We first record a general observation that will be used twice in the
proof. Suppose that $(W_x)_{x\in\N}$ is a uniformly c.e.\
$A$-homogeneous family and that every $W_x$ is infinite.

Fix a computable enumeration without repetitions
\[
   (x_0,n_0),(x_1,n_1),\ldots
\]
of $\N^2$. We construct a function
\[
   G:\N^2\to\N
\]
recursively along this enumeration. When treating $(x_i,n_i)$, wait
for an element of $W_{x_i}$ which has not been used at any earlier
step, and define $G(x_i,n_i)$ to be the first such element found.

Since $W_{x_i}$ is infinite and only finitely many values have been
used before step $i$, this search terminates. Thus $G$ is a total
computable injection satisfying
\[
   G(x,n)\in W_x
\]
for all $x,n$. By the $A$-homogeneity of $(W_x)_{x\in\N}$,
\[
   A(G(x,n))=A(x).
\]
Consequently, the map
\[
   \langle x,n\rangle\longmapsto G(x,n)
\]
witnesses
\[
   \Cyl(A)\leq_1 A.
\]
We now distinguish the two cases $Z=\varnothing$ and
$Z\neq\varnothing$.

If $Z=\varnothing$, then every $V_x$ is infinite. Applying the
observation above with $W_x=V_x$ gives
\[
   \Cyl(A)\leq_1 A.
\]

Since
\[
   x\longmapsto\langle x,0\rangle
\]
is a one-one reduction from $A$ to $\Cyl(A)$, it follows that
$A\equiv_1\Cyl(A)$, contradicting the assumption that $A$ is not a
cylinder. Thus we may assume that $Z\neq\varnothing$.

For $z\in Z$, let
\[
   X_z=\{x:z\in V_x\},
\qquad
   D_Z=\bigcup_{z\in Z}X_z.
\]
Each $X_z$ is c.e., and homogeneity gives
\[
   X_z\subseteq
   \begin{cases}
      A,&z\in A,\\
      \overline A,&z\notin A.
   \end{cases}
\]

Every $x\notin D_Z$ has $V_x\cap Z=\varnothing$, so $V_x$ is
infinite.

Suppose first that, for every $z\in Z$, there is some
$d_z\notin D_Z$ such that
\[
   A(d_z)=A(z).
\]
Since $Z$ is finite, the finitely many parameters $d_z$ may be
hardcoded.

Fix a uniformly computable increasing approximation
\[
   V_{x,0}\subseteq V_{x,1}\subseteq\cdots,
   \qquad
   V_x=\bigcup_{s\in\N}V_{x,s},
\]
to the uniformly c.e.\ family $(V_x)_{x\in\N}$. Define
\[
   W_{x,s}
   =
   V_{x,s}
   \cup
   \bigcup_{\substack{z\in Z\\ z\in V_{x,s}}}V_{d_z,s},
   \qquad
   W_x=\bigcup_{s\in\N}W_{x,s}.
\]
Thus
\[
   W_x
   =
   V_x
   \cup
   \bigcup_{\substack{z\in Z\\ z\in V_x}}V_{d_z},
\]
and the family $(W_x)_{x\in\N}$ is uniformly c.e.

The family $(W_x)_{x\in\N}$ is also $A$-homogeneous.
Indeed, let $u\in W_x$. If $u\in V_x$, then
\[
   A(u)=A(x)
\]
by the $A$-homogeneity of $(V_x)_{x\in\N}$. Otherwise,
\[
   u\in V_{d_z}
   \qquad\text{for some }z\in Z\cap V_x.
\]
Then
\[
   A(u)=A(d_z)=A(z)=A(x).
\]

Finally, every $W_x$ is infinite. If
\[
   V_x\cap Z=\varnothing,
\]
then $V_x$ is infinite by our contradictory assumption, and hence so
is $W_x$. Otherwise, choose some
\[
   z\in V_x\cap Z.
\]

Since $d_z\notin D_Z$, for every $w\in Z$ we have
\[
   d_z\notin X_w.
\]

By the definition of $X_w$, this means that
\[
   w\notin V_{d_z}
   \qquad\text{for every }w\in Z,
\]
and therefore
\[
   V_{d_z}\cap Z=\varnothing.
\]
Our contradictory assumption now implies that $V_{d_z}$ is infinite.
Since
\[
   V_{d_z}\subseteq W_x,
\]
the set $W_x$ is infinite as well.

Applying the observation above to $(W_x)_{x\in\N}$ gives
\[
   \Cyl(A)\leq_1 A.
\]
Since the map
\[
   x\longmapsto\langle x,0\rangle
\]
witnesses
\[
   A\leq_1\Cyl(A),
\]
it follows that
\[
   A\equiv_1\Cyl(A),
\]
contradicting the assumption that $A$ is not a cylinder.

Therefore the preceding alternative cannot occur. Hence there is some
$z\in Z$ such that
\[
   A(d)=A(z)\quad\Longrightarrow\quad d\in D_Z
\]
for every $d\in\N$. Let
\[
   c=A(z).
\]

Since $Z$ is finite, the finite set
\[
   \{w\in Z:A(w)=c\}
\]
may be hardcoded.

Then
\[
   \{x\in\N:A(x)=c\}
   =
   \bigcup_{\substack{w\in Z\\A(w)=c}}X_w.
\]

Indeed, if $x\in X_w$ for some $w\in Z$ with $A(w)=c$, then
homogeneity gives $A(x)=A(w)=c$. Conversely, suppose that
$A(x)=c$. By the choice of $z$, we have $x\in D_Z$, so
$x\in X_w$ for some $w\in Z$. Since $w\in V_x$, homogeneity gives
$A(w)=A(x)=c$. Hence $x$ belongs to the displayed union.

The displayed union is a finite union of c.e.\ sets and is therefore
c.e. Since it equals $A$ when $c=1$ and $\overline A$ when $c=0$,
this contradicts the assumption that neither $A$ nor $\overline A$
is c.e.

\end{proof}

\begin{definition}
A partial computable function $\varphi$ is a \emph{partial ascending
self-reduction of $A$} if, for every $x\in\dom\varphi$,
\[
   \varphi(x)>x
   \qquad\text{and}\qquad
   A(\varphi(x))=A(x).
\]
\end{definition}

\begin{lemma}[Cofinite-domain lemma]\label{lem:cofinite}
Assume that $A$ is nonrecursive and is not a cylinder. If $\varphi$ is
a partial ascending self-reduction of $A$, then
\[
   \N\setminus\dom\varphi
\]
is infinite.
\end{lemma}

\begin{proof}

Suppose that $F=\N\setminus\dom\varphi$ is finite.
Since $A$ is nonrecursive, both $A$ and $\overline A$ are infinite.
For each $x\in F$, choose $t_x>x$ with $A(t_x)=A(x)$.
Since $F$ is finite, both $F$ and the finite table
$x\mapsto t_x$ may be hardcoded. Define
\[
   g(x)=
   \begin{cases}
      t_x,&x\in F,\\
      \varphi(x),&x\notin F.
   \end{cases}
\]
Since $F$ is finite, membership in $F$ is decidable after hardcoding
its elements. Moreover, $\varphi(x)$ converges whenever $x\notin F$.
Hence $g$ is total computable and
\[
   g(x)>x,\qquad A(g(x))=A(x)
\]
for every $x$.

Fix a computable enumeration without repetitions
\[
   (x_0,n_0),(x_1,n_1),\ldots
\]
of $\N^2$. We define $G$ recursively along this enumeration.
Suppose that $G(x_j,n_j)$ has already been defined for every $j<i$,
and let
\[
   U_i=\{G(x_j,n_j):j<i\}.
\]
The set $U_i$ is finite. Since
\[
   x_i<g(x_i)<g^2(x_i)<\cdots,
\]
the orbit $\{g^k(x_i):k\in\N\}$ is infinite. Hence there is some
$k$ such that
\[
   g^k(x_i)\notin U_i.
\]
Let
\[
   k_i=\mu k\,[g^k(x_i)\notin U_i]
\]
and define
\[
   G(x_i,n_i)=g^{k_i}(x_i).
\]

The search defining $k_i$ is effective because $g$ is total
computable, and it terminates because the orbit of $x_i$ is infinite
whereas $U_i$ is finite. Thus $G:\N^2\to\N$ is total computable.
It is injective by construction. Furthermore, an induction on $k$
gives
\[
   A(g^k(x))=A(x)
\]
for all $x,k$, and hence
\[
   A(G(x,n))=A(x).
\]
Therefore the function
\[
   \langle x,n\rangle\longmapsto G(x,n)
\]
witnesses
\[
   \Cyl(A)\leq_1 A.
\]
Since
\[
   x\longmapsto\langle x,0\rangle
\]
witnesses $A\leq_1\Cyl(A)$, it follows that
\[
   A\equiv_1\Cyl(A).
\]
Thus $A$ is a cylinder, a contradiction.
\end{proof}

\section{The priority construction}

The local interpolation theorem below is proved by a priority construction
relative to $0'$. We first build a $\Sigma^0_2$ set $N$ whose
complement will later be realized as the infinite-fibre profile of a
computable surjection. The two lemmas of the preceding section provide
the avoidance and infinitude properties needed to meet the two families
of priority requirements.

\begin{theorem}[Local interpolation]\label{thm:local}
Assume that $A$ is nonrecursive, is not a cylinder, and neither $A$ nor
$\overline A$ is c.e. Then there is a set $C\equiv_{\m} A$ such that
\[
   [A]_{\fo}<[C]_{\fo}<[\Cyl(A)]_{\fo}.
\]
\end{theorem}

\begin{proof}
Fix an effective enumeration
\[
   \bigl((V_x^e)_{x\in\N}\bigr)_{e\in\N}
\]
of all uniformly c.e.\ families, and an effective enumeration
$(\varphi_e)_{e\in\N}$ of all partial computable functions.

Fix standard finite stage approximations
\[
   V^e_{x,0}\subseteq V^e_{x,1}\subseteq\cdots,
   \qquad
   V_x^e=\bigcup_s V^e_{x,s},
\]
uniformly computable in $e,x,s$.

We construct a set $N$ which is c.e.\ in $0'$, satisfying the
requirements
\[
\mathcal U_e:\quad
   (V_x^e)_x\text{ is $A$-homogeneous}
   \Longrightarrow
   \exists x\,
   [V_x^e\text{ finite and }V_x^e\subseteq N],
\]
and
\[
\mathcal L_e:\quad
   \varphi_e\text{ is a partial ascending self-reduction of }A
   \Longrightarrow
   \exists m\,
   [m\notin N\cup\dom\varphi_e].
\]
Use the priority ordering
\[
   \mathcal U_0>\mathcal L_0>
   \mathcal U_1>\mathcal L_1>\cdots.
\]
The construction does not attempt to decide whether either antecedent
in the requirements holds. Its actions are independent of $A$.
The hypotheses involving $A$ are used only in the verification:
Lemma~\ref{lem:avoid} is used to verify the
$\mathcal U_e$-requirements, while Lemma~\ref{lem:cofinite} is used
to verify the $\mathcal L_e$-requirements.

At stage $s$, the first $s$ requirements are visited in the
displayed priority order.

Each $\mathcal L_e$ may carry one marker $m_e$. Initially,
$N=\varnothing$ and no requirement carries a marker.

At every point of the construction we maintain the invariant that
all active markers are pairwise distinct and lie outside the current
content of $N$. All enumerations into $N$ and all cancellations take
effect immediately, before the next requirement is visited.

The intended roles of the requirements are as follows.
After $P=\N\setminus N$, the map $p$, and $C=p^{-1}(A)$ have been
defined, the $\mathcal U_e$ requirements will rule out finite-one
reductions from $\Cyl(A)$ to $C$, while the $\mathcal L_e$
requirements will rule out finite-one reductions from $C$ to $A$.

\medskip
\noindent
\textbf{Strategy for $\mathcal U_e$.}
Let $Z_e$ be the finite set of markers currently carried by
higher-priority $\mathcal L$-requirements. If $\mathcal U_e$ has not
yet acted, search among pairs $x,t\leq s$ for one such that
\[
   V^e_{x,t}\cap Z_e=\varnothing,
   \qquad
   V^e_{x,t}\subseteq\{0,\ldots,s\},
\]
and
\[
   \forall r\geq t\,
   [V^e_{x,r}=V^e_{x,t}].
\]
The last condition is decidable in $0'$.
Indeed, for fixed $e,x,t$, failure of the stability condition is the
$\Sigma^0_1$ statement
\[
   \exists r\geq t\,
   [V^e_{x,r}\neq V^e_{x,t}].
\]
Hence $0'$ decides whether the row receives any new element after the
specified stage $t$. Notice that this does not amount to deciding,
with $0'$, whether an arbitrary c.e.\ row is finite.

If such a pair exists,
choose the least one in a fixed computable ordering and let
\[
   F=V^e_{x,t}.
\]
By the stability condition, $F=V_x^e$.

First cancel every active marker $m_i$ carried by a lower-priority
$\mathcal L_i$ such that
\[
   m_i\in F.
\]
Then enumerate every element of $F$ into $N$. Since
\[
   F\cap Z_e=\varnothing,
\]
no higher-priority marker belongs to $F$, while every lower-priority
marker belonging to $F$ has already been cancelled. Hence, after
the enumeration, every marker which remains active lies outside the
current content of $N$. Declare $\mathcal U_e$ permanently satisfied.

If no such pair exists, do nothing.

\medskip
\noindent
\textbf{Strategy for $\mathcal L_e$.}
If $\mathcal L_e$ currently has no marker, search among $m\leq s$
for an $m$ such that
\[
   m\notin\dom\varphi_e,
   \qquad
   m\notin N_s,
\]
and $m$ is distinct from all currently active markers, where $N_s$
denotes the current finite content of $N$ at the moment when
$\mathcal L_e$ is visited. The condition
$m\notin\dom\varphi_e$ is decidable in $0'$.

If such numbers exist, choose the least one and designate it as
$m_e$; otherwise do nothing. If a higher-priority
$\mathcal U$-requirement later selects a row containing $m_e$, cancel
$m_e$ before that row is enumerated into $N$. We call this
cancellation an injury to the $\mathcal L_e$-strategy. Every
lower-priority $\mathcal U$-requirement which has not yet acted
includes $m_e$ among its restraints.
We now verify that all requirements are satisfied.

\medskip
\noindent
\textbf{Verification.}
Observe that an active marker carried by an $\mathcal L_e$-strategy
can be injured only by a higher-priority $\mathcal U$-strategy.
Indeed, a lower-priority $\mathcal U$-strategy either had already
acted when the marker was chosen, in which case its stabilized row
was already contained in $N$ and hence could not contain the marker,
or had not yet acted, in which case, as long as the marker remains
active, it includes the marker among its restraints and therefore
must choose a row disjoint from it.

By definition, once the $\mathcal U_i$-strategy acts, it declares
$\mathcal U_i$ permanently satisfied and is never eligible to act
again. Indeed, its chosen stabilized row is then entirely contained
in $N$. 

Thus every $\mathcal U_i$ acts at most once. Hence, for every $e$,
there is a stage $s_e$ after which no $\mathcal U_i$ of higher
priority than $\mathcal L_e$ acts. By the preceding observation,
$\mathcal L_e$ cannot be injured after stage $s_e$.

If $\mathcal L_e$ carries a marker at the end of stage $s_e$, that
marker is final. Otherwise, either it chooses a marker at some later
visit, in which case that marker is final, or it remains permanently
markerless. Thus every $\mathcal L_e$ eventually reaches a final
state. Consequently, for every fixed requirement, the finite set of
active markers carried by higher-priority $\mathcal L$-strategies
eventually stabilizes.

Suppose first that $\varphi_e$ is a partial ascending self-reduction.
By Lemma~\ref{lem:cofinite},
\[
   D_e=\N\setminus\dom\varphi_e
\]
is infinite. Choose a stage $s_0$ after which no higher-priority
$\mathcal U$-strategy acts and the finite set of active markers
carried by higher-priority $\mathcal L$-strategies has stabilized.

Choose
\[
   m\in D_e
\]
so large that $m>s_0$, that $m$ is larger than every final
higher-priority marker, and that $\mathcal L_e$ is among the
requirements visited at stage $m$. Consider stage $m$.

Before $\mathcal L_e$ is visited at stage $m$, we have $m\notin N$.
Indeed, an action performed at an earlier stage $r<m$ enumerates only
numbers at most $r$, and therefore cannot enumerate $m$. At stage
$m$, before $\mathcal L_e$ is visited, only higher-priority
requirements have been visited. The higher-priority
$\mathcal L$-requirements do not enumerate into $N$, and no
higher-priority $\mathcal U$-requirement acts after $s_0$. Hence no
action before the visit of $\mathcal L_e$ at stage $m$ can have placed
$m$ into $N$.

Moreover, $m$ is not an active marker when $\mathcal L_e$ is visited.
All higher-priority active markers have already reached their final
values and are smaller than $m$. Any active marker carried by a
lower-priority $\mathcal L$-strategy was chosen at some earlier stage
$r<m$, since no lower-priority requirement has yet been visited during
stage $m$; hence that marker is at most $r<m$.

Any marker already carried by $\mathcal L_e$ itself was also chosen
at an earlier stage $r<m$, and hence is smaller than $m$.

If $\mathcal L_e$ already carries a marker when it is visited at
stage $m$, let $m_e$ denote this marker. By construction,
\[
   m_e\notin\dom\varphi_e.
\]
No higher-priority $\mathcal U$-requirement acts after $s_0$.
Every lower-priority $\mathcal U$-requirement which had not yet acted
when $m_e$ was chosen includes $m_e$ among its restraints and hence,
if it acts later, selects a row disjoint from $\{m_e\}$.

On the other hand, if a lower-priority $\mathcal U$-requirement had
already acted before $m_e$ was chosen, its selected row was already
contained in $N$ at the time of the choice. Since $m_e$ was chosen
outside the current content of $N$, that row cannot contain $m_e$.
Therefore
\[
   m_e\notin N\cup\dom\varphi_e,
\]
and $\mathcal L_e$ is satisfied.

Otherwise, $\mathcal L_e$ has no marker when it is visited at stage
$m$. By the choice of $m$ and the preceding observations, $m$ itself
is outside $N$, outside $\dom\varphi_e$, and distinct from all active
markers. Thus $m$ is a candidate in the bounded search performed by
$\mathcal L_e$, and hence the strategy chooses some marker $m_e$.

No higher-priority $\mathcal U$-requirement acts thereafter, and every
lower-priority $\mathcal U$-requirement which has not yet acted treats
$m_e$ as a restraint. If a lower-priority $\mathcal U$-requirement
acted before $m_e$ was chosen, its selected row is already contained
in $N$, whereas $m_e$ was chosen outside the current content of $N$.
Hence that row cannot contain $m_e$. Therefore
\[
   m_e\notin N\cup\dom\varphi_e.
\]

If a marker belonging to a higher-priority $\mathcal L$-strategy is
later defined or redefined, its new
value is chosen outside the current content of $N$. Since every row
previously selected by a $\mathcal U$-strategy is already contained
in $N$, the new marker cannot belong to any such row. Thus a previous
$\mathcal U$-action remains compatible with every later definition or
redefinition of a higher-priority marker.

Now suppose that $(V_x^e)_x$ is $A$-homogeneous. If
$\mathcal U_e$ has already acted, it is permanently satisfied.
Otherwise, after all higher-priority marker restraints have reached
their final values, let $Z$ be the resulting finite set of markers.
At every subsequent visit of $\mathcal U_e$, the set $Z_e$ used by
the strategy is therefore equal to $Z$.

By Lemma~\ref{lem:avoid}, there is an $x$ such that
\[
   V_x^e\text{ is finite}
   \qquad\text{and}\qquad
   V_x^e\cap Z=\varnothing.
\]

Since $V_x^e$ is finite, choose a stage $t$ by which every element of
$V_x^e$ has been enumerated and after which no new element enters the
row. Thus
\[
   V^e_{x,t}=V_x^e
\]
and
\[
   V^e_{x,r}=V^e_{x,t}
   \qquad\text{for every }r\geq t.
\]

For every
sufficiently large stage $s$ at which $\mathcal U_e$ is visited, we
have
\[
   x,t\leq s,
   \qquad
   V^e_{x,t}\cap Z_e=\varnothing,
   \qquad
   V^e_{x,t}\subseteq\{0,\ldots,s\}.
\]
Thus $(x,t)$ satisfies all the tests in the $\mathcal U_e$-strategy.
Since every fixed requirement is visited at all sufficiently large
stages, the bounded $0'$-search eventually makes $\mathcal U_e$ act.
Therefore every $\mathcal U_e$ is satisfied.

At each stage only finitely many requirements are visited, each
strategy performs only finitely many $0'$-decidable tests, and only a
finite set is enumerated into $N$. Thus the enumeration of $N$ is
computable in $0'$.

Let
\[
   P=\N\setminus N.
\]

Since $N$ is c.e.\ in $0'$, we have $N\in\Sigma^0_2$.
Therefore $P=\mathbb N\setminus N$ belongs to $\Pi^0_2$.

Although $N$ was enumerated relative to $0'$, the profile-realization
lemma converts the resulting $\Pi^0_2$ set $P$ into the infinite-fibre
profile of an ordinary computable function.
By Lemma~\ref{lem:profile}, fix a total computable
surjection $p$ with computable section such that
\[
   I_p=P,
\]
and put
\[
   C=p^{-1}(A).
\]
Since
\[
   I_p=P=\N\setminus N,
\]
we have
\[
   y\in P
   \quad\Longleftrightarrow\quad
   p^{-1}(y)\text{ is infinite},
\]
and
\[
   y\in N
   \quad\Longleftrightarrow\quad
   p^{-1}(y)\text{ is finite}.
\]
This is the link between the priority requirements and the
finite-one degree of $C$: the $\mathcal U_e$ requirements exploit
finite fibres over $N$, whereas the $\mathcal L_e$ requirements
exploit infinite fibres over $P$.

Let $s$ be a computable section of $p$. For every $y\in\N$,
\[
   y\in A
   \quad\Longleftrightarrow\quad
   p(s(y))\in A
   \quad\Longleftrightarrow\quad
   s(y)\in C.
\]
Since $s$ is injective, it witnesses
\[
   A\leq_1 C,
\]
and hence
\[
   A\leq_{\fo}C.
\]
On the other hand, for every $u\in\N$,
\[
   u\in C
   \quad\Longleftrightarrow\quad
   p(u)\in A,
\]
so $p$ witnesses
\[
   C\leq_{\m}A.
\]
Since $A\leq_1C$ implies $A\leq_{\m}C$, it follows that
\[
   C\equiv_{\m}A.
\]

We first show that
\[
   C<_{\fo}\Cyl(A).
\]

There is a one-one reduction
\[
   u\longmapsto \langle p(u),u\rangle
\]
from $C$ to $\Cyl(A)$. Suppose conversely that
\[
   h:\Cyl(A)\leq_{\fo} C.
\]
For each $x$, define
\[
   V_x=\{p(h(\langle x,n\rangle)):n\in\N\}.
\]
Although $p$ was chosen only after the construction of $N$, there is
no circularity here. Once $p$ and the hypothetical reduction $h$ are
fixed, $(V_x)_{x\in\N}$ is an ordinary uniformly c.e.\ family.
Consequently, there is an index $e$ such that
\[
   V_x=V_x^e
   \qquad\text{for every }x.
\]

This family is $A$-homogeneous. Indeed, if
\[
   y=p(h(\langle x,n\rangle)),
\]
then
\[
   y\in A
   \iff h(\langle x,n\rangle)\in C
   \iff \langle x,n\rangle\in\Cyl(A)
   \iff x\in A.
\]
Hence requirement $\mathcal U_e$, for the index $e$ fixed above,
supplies an $x$ for which $V_x=V_x^e$ is finite and
\[
   V_x\subseteq N=\N\setminus P.
\]

For every $y\in V_x$, the fibre $p^{-1}(y)$ is finite. Therefore
\[
   h(\{\langle x,n\rangle:n\in\N\})
   \subseteq
   \bigcup_{y\in V_x}p^{-1}(y),
\]
a finite set.
This is impossible, since a finite-one function cannot map an
infinite set into a finite set.

Thus
\[
   C<_{\fo}\Cyl(A).
\]

We have already shown that
\[
   A\leq_{\fo}C.
\]
It remains to prove that
\[
   C\not\leq_{\fo}A.
\]
Suppose, toward a contradiction, that
\[
   h:C\leq_{\fo}A.
\]

Define the partial computable function $f_h$ by
\[
   f_h(y)=h(u),
\]

where $u$ is the least number such that
\[
   p(u)=y
   \qquad\text{and}\qquad
   h(u)>y.
\]

Whenever $f_h(y)$ is defined,
\[
   f_h(y)>y
\]
and, because $h$ is a reduction,
\[
   A(f_h(y))
   =A(h(u))
   =A(p(u))
   =A(y).
\]
Thus $f_h$ is a partial ascending self-reduction of $A$.

If $y\in P$, then $p^{-1}(y)$ is infinite. Since $h$ is finite-one,
\[
   h^{-1}(\{0,\ldots,y\})
   =
   \bigcup_{z\leq y}h^{-1}(z)
\]
is finite. Hence some $u\in p^{-1}(y)$ satisfies $h(u)>y$, and so
\[
   P\subseteq\dom f_h.
\]
Again, there is no circularity in applying one of the requirements
fixed before the construction. Once $p$ and the hypothetical
reduction $h$ are fixed, $f_h$ is an ordinary partial computable
function. Hence there is an index $e$ such that
\[
   \varphi_e=f_h.
\]
By the verification of $\mathcal L_e$, there is a final marker
$m_e$ such that
\[
   m_e\notin N
   \qquad\text{and}\qquad
   m_e\notin\dom f_h.
\]

But $m_e\notin N$ means $m_e\in P$, contradicting
$P\subseteq\dom f_h$.

Hence
\[
   C\not\leq_{\fo}A.
\]
Since $A\leq_{\fo}C$, it follows that
\[
   A<_{\fo}C.
\]
This proves the theorem.
\end{proof}

\begin{remark}

Several points concerning the effectiveness of the construction are
worth emphasizing.

First, for a partial computable function $\varphi_e$, the condition
\[
   m\notin\dom\varphi_e
\]
is decidable in $0'$. Thus an $\mathcal L_e$-strategy can choose a
marker with permanent certification that it lies outside
$\dom\varphi_e$.

Second, a $\mathcal U$-requirement acting at stage $s$ enumerates into
$N$ only numbers at most $s$. Consequently, no action before stage
$m$ can enumerate $m$ into $N$. Once the higher-priority
$\mathcal U$-requirements have ceased acting, a newly chosen marker
is also protected from every lower-priority $\mathcal U$-requirement.
One which has already acted selected a row already contained in $N$;
since the marker is chosen outside the current content of $N$, that
row cannot contain it. One which has not yet acted includes the
marker among its restraints and therefore selects a row disjoint
from it.

Finally, the finite parameters hardcoded in the proofs of
Lemmas~\ref{lem:avoid} and~\ref{lem:cofinite} need not be recovered
by the priority construction. For a $\mathcal U_e$-requirement, the
current finite restraint set $Z_e$ is explicitly available at every
stage. Once it stabilizes, Lemma~\ref{lem:avoid} guarantees the
existence of a finite row disjoint from it, and the bounded
$0'$-search eventually finds a stabilization pair $(x,t)$ without
having to recognize that the restraints have stabilized. For an
$\mathcal L_e$-requirement, Lemma~\ref{lem:cofinite} is used only to
guarantee arbitrarily large elements of
\[
   \N\setminus\dom\varphi_e,
\]
while each individual candidate is checked by the $0'$-test described
above. Thus the two lemmas are used only existentially in the
verification, and no uniform computation of the finite parameters
appearing in their proofs is required.

\end{remark}

\section{The dichotomy}

We now combine the local interpolation theorem with Maslova's
dichotomy for nonrecursive c.e.\ many-one degrees to obtain the
global one-or-infinity result.

We use the following consequence of Maslova's
theorem~\cite{Maslova1979}, in the formulation recorded by Richter,
Stephan, and Zhang~\cite[p.~17]{RSZ}: every nonrecursive c.e.\ many-one
degree contains either exactly one or infinitely many finite-one
degrees.

\begin{corollary}\label{cor:oq2}
Every nonrecursive many-one degree contains either exactly one or
infinitely many finite-one degrees. In particular, there is no
nonrecursive many-one degree containing at least two but only finitely
many finite-one degrees.
\end{corollary}

\begin{proof}
For a many-one degree $\mathbf d$, let
\[
   \mathcal F(\mathbf d)
   =
   \{[X]_{\fo}:X\in\mathbf d\}.
\]
Suppose, toward a contradiction, that $\mathbf d$ is nonrecursive and
that
\[
   1<|\mathcal F(\mathbf d)|<\aleph_0.
\]

Let $T$ be the greatest element of $\mathcal F(\mathbf d)$, that is
\[
T = [\Cyl(X)]_{\fo}
\]
for any $X\in \mathbf d$.
Since $\mathcal F(\mathbf d)\setminus\{T\}$ is finite and nonempty,
there is an element $E<T$ which is covered by $T$; that is, there is
no $D\in\mathcal F(\mathbf d)$ such that
\[
   E<D<T.
\]

Choose a set $A\in\mathbf d$ such that
\[
   [A]_{\fo}=E.
\]

Since $A\in\mathbf d$ and $\mathbf d$ is nonrecursive, 
$A$ is nonrecursive. Moreover, $A$ is not a cylinder.
Indeed, if $A$ were a cylinder, then
\[
   E=[A]_{\fo}=[\Cyl(A)]_{\fo}=T,
\]
because $[\Cyl(A)]_{\fo}$ is the greatest finite-one degree inside
$[A]_{\m}=\mathbf d$. This contradicts $E<T$.

If $A$ is c.e., then $\mathbf d=[A]_{\m}$ is a nonrecursive c.e.\
many-one degree containing finitely many but more than one finite-one
degrees, contrary to Maslova's theorem.

Suppose next that $\overline A$ is c.e. Put
\[
   \mathbf d^{\,c}=[\overline A]_{\m}.
\]
Complementation induces an order isomorphism
\[
   \Phi:\mathcal F(\mathbf d)\longrightarrow
   \mathcal F(\mathbf d^{\,c})
\]
given by
\[
   \Phi([X]_{\fo})=[\overline X]_{\fo}.
\]
Indeed, for every total computable function $f$,
\[
   f:X\leq_{\m}Y
   \quad\Longleftrightarrow\quad
   f:\overline X\leq_{\m}\overline Y,
\]
and, since the fibres of $f$ are unchanged,
\[
   f:X\leq_{\fo}Y
   \quad\Longleftrightarrow\quad
   f:\overline X\leq_{\fo}\overline Y.
\]

Thus $\Phi$ is well defined and order preserving. The same
complementation map, now considered from
$\mathcal F(\mathbf d^{\,c})$ to $\mathcal F(\mathbf d)$, is the
inverse of $\Phi$. Hence $\Phi$ is an order isomorphism.

Consequently,
\[
   |\mathcal F(\mathbf d^{\,c})|
   =
   |\mathcal F(\mathbf d)|,
\]
so $\mathbf d^{\,c}$ also contains finitely many but more than one
finite-one degrees. Since $\overline A$ is c.e.\ and nonrecursive,
this again contradicts Maslova's theorem.

We may therefore assume that neither $A$ nor $\overline A$ is c.e.
Theorem~\ref{thm:local} now yields a set $C\equiv_{\m} A$ such that
\[
   [A]_{\fo}<[C]_{\fo}<[\Cyl(A)]_{\fo}.
\]
Since $C\equiv_{\m}A$, the degree $[C]_{\fo}$ belongs to
$\mathcal F(\mathbf d)$. Moreover,
\[
   [A]_{\fo}=E
   \qquad\text{and}\qquad
   [\Cyl(A)]_{\fo}=T.
\]
Hence
\[
   E<[C]_{\fo}<T,
\]
contradicting the fact that $E$ is covered by $T$.

Therefore $\mathcal F(\mathbf d)$ cannot be finite of cardinality
greater than one. Since every many-one degree contains at least one
finite-one degree, it follows that every nonrecursive many-one degree
contains either exactly one or infinitely many finite-one degrees.
\end{proof}

\section{Conclusion}

We have shown that every nonrecursive many-one degree contains either
exactly one or infinitely many finite-one degrees, giving a negative
answer to Open Question 2 of Richter--Stephan--Zhang.

The main ingredient is the local interpolation theorem: whenever $A$
is nonrecursive, is not a cylinder, and neither $A$ nor $\overline A$
is c.e., the finite-one degree of $A$ can be strictly interpolated
below the greatest finite-one degree in $[A]_{\m}$. Together with
Maslova's theorem for c.e.\ many-one degrees and complementation, this
local phenomenon rules out every finite nontrivial configuration of
finite-one degrees inside a nonrecursive many-one degree.

The local interpolation theorem also raises finer questions about the
internal order structure of finite-one degrees inside a fixed
many-one degree. In particular, it is natural to ask whether stronger
interpolation properties hold between arbitrary comparable
finite-one degrees, whether the corresponding intervals can be dense,
and whether the nonrecursive many-one degrees containing exactly one
finite-one degree admit a structural characterization.

\section*{Acknowledgements}

The main result of this paper was discovered by ChatGPT 5.6 Sol (OpenAI).

The author has reworked and verified all arguments and bears sole responsibility for the correctness of the results.

\end{document}